\documentclass[10pt]{article}   

\usepackage{graphicx}
\usepackage{mathrsfs} 

\usepackage{hyperref, amsrefs, color}
\usepackage{graphicx}
\usepackage[font=small]{caption} 
\usepackage{float}
\usepackage{subcaption}
\usepackage{wrapfig}
\usepackage{enumerate}

\hypersetup{colorlinks, citecolor=blue, linkcolor=blue, urlcolor=blue} 
\usepackage[margin=1.2in]{geometry} 

\usepackage{xfrac}

\usepackage{amsmath, amsthm, latexsym, amssymb, amsfonts}  
\usepackage{bbm, dsfont}

\usepackage{verbatim}

\newcommand{\R}{\mathbb R}

\newcommand{\e}{\epsilon}
\newcommand{\ep}{\varepsilon}

\newcommand{\LP}{\left( }
\newcommand{\RP}{\right) }

\newcommand{\sgn}{\text{sgn}}
\newcommand{\supp}{\text{supp}}

\newcommand{\bgamma}{\boldsymbol{\gamma}}

\newcommand{\tPhi}{\tilde{\Phi}}

\newcommand{\wm}{|w|_-}

\newtheorem{Th}{Theorem}[section]

\newtheorem{Prop}[Th]{Proposition} 
\newtheorem{Lemma}[Th]{Lemma}
\theoremstyle{definition} 
\newtheorem{Rem}[Th]{Remark}

\newtheorem{Assump}[Th]{Assumption}

\usepackage{tikz}
\usetikzlibrary{decorations.pathreplacing}
\usetikzlibrary{decorations.pathmorphing}
\usetikzlibrary{arrows}
\usetikzlibrary{calc}
\usetikzlibrary{patterns}
\usetikzlibrary{matrix}
\usepgflibrary{arrows}
\usetikzlibrary{arrows}
\usetikzlibrary{decorations.markings}
\tikzstyle{vertex}=[circle, draw, inner sep=0pt, minimum size=6pt]

 \title{The incompressible limit of an inhomogeneous model of tissue growth}
 \author{Anthony Sulak, Olga Turanova}
 \date{\today}
\begin{document}

 \maketitle  

\begin{abstract}
We study a  porous medium equation that models tissue growth in a heterogeneous environment. 
We show that, in the incompressible limit, solutions converge to those of a weak form of a Hele-Shaw type free boundary problem. 
To obtain enough compactness to take the limit, we establish an $L^4$ bound on the gradient of the pressure and an estimate of Aronson-B\'{e}nilan type.

\end{abstract}

\section{Introduction}
The focus of our work is the $m\rightarrow\infty$ limit (called the incompressible limit or stiff pressure limit) of the inhomogeneous porous medium equation with reaction, 
\begin{align} \label{eqn: PME}
         \partial _t u_m = \nabla \cdot \LP \frac{u_m}{a(x,t)} \nabla p_m \RP  + \frac{u_m}{a(x,t)}\Phi (x,t,p_m) \text{ on }\R^d\times (0,\infty),
\end{align}
where the pressure $p_m$ is given in terms of the density $u_m$ by the constitutive law,
\begin{equation}
    \label{eq:constitutive}
p_m = \frac{m}{m-1} \LP \frac{u_m }{b(x,t)} \RP ^ {m-1}, \quad m\geq 2.
\end{equation}
The coefficients $a$ and $b$ are assumed to be bounded from above and strictly away from zero and satisfy certain structural assumptions. It is also assumed  that the growth term $\Phi$  is strictly decreasing in $p$ and that there exists $p_M>0$ with $\Phi(x,t, p_M)=0$, which corresponds to a ceiling on the maximum pressure that the medium can support.  (The precise assumptions are listed in Section \ref{ss:assumptions}.)  In addition, we remark that \eqref{eqn: PME} may also be written as,
\begin{align} \label{eqn: PME density}
    \partial _t u_m = \nabla \cdot \LP \frac{b}{a} \nabla \LP \frac{u_m}{b} \RP^m \RP  + \frac{u_m}{a}\Phi (x,t,p_m). 
\end{align}

The homogeneous version of this model ($a\equiv b\equiv 1$ and $\Phi$ depending only on $p$) was proposed by Byrne and Drasdo in \cite{ByrneDrasdo} as a model for tumor growth. Its incompressible limit was shown by Perthame, Quir\`{o}s, and V\'{a}zquez in \cite{PQV2014HS} to be a free boundary problem of Hele-Shaw type. 
Since then, there has been a large number of works on models of this type: see Section \ref{ss:lit} for a  brief literature review. Here, the presence of the nonconstant coefficients $a$ and $b$   represents heterogeneity in the underlying medium and in the cellular packing density, respectively \cite{TangVaucheletCheddadietal, Vazquez2007PME}.

In our main result, Theorem \ref{Prop: main prop}, we establish that, as $m\rightarrow \infty$, the density and pressure converge to the pair $(u_\infty, p_\infty)$, which is the (unique) weak solution of 
\begin{equation}
\label{eq:limiting eq}
\left\{
    \begin{split}&
    \partial _t  u_\infty = \nabla \cdot \LP \frac{b}{a}  \nabla p_\infty \RP + \frac{u_\infty}{a} \Phi(x,t,p_\infty)\\ 
    & \left(1-\frac{u_\infty}{b}\right)p_\infty = 0 \text{ almost everywhere.} 
    \end{split}
\right.
\end{equation}

In addition, we provide more detail on the behavior of the limiting density and pressure. The heuristics for this can be seen by examining the equation that the pressure satisfies,
\begin{equation}
\label{eqn: PME Pressure 2}
\partial _t p_m  - 
    \frac{\vert \nabla p_m \vert ^2}{a} 
 =(m-1) \frac{p_m}{b} \LP
    \nabla \cdot\left(\frac{b}{a}\nabla p_m\right) 
    + \frac{b}{a}\Phi (x,t,p_m) -\partial_tb \RP,
    \end{equation}
which is obtained by multiplying (\ref{eqn: PME density}) by $\frac{m}{b}\left(\frac{u_m}{b}\right)^{m-2}$ and performing standard manipulations. It is natural to guess that, in the $m\rightarrow\infty$ limit of (\ref{eqn: PME Pressure 2}), the  righthand side of (\ref{eqn: PME Pressure 2}) converges to zero: 
\begin{equation}
        \label{eq:complementarity}
\frac{p_\infty}{b}  \LP \nabla\cdot\left(\frac{b}{a}\nabla p_\infty\right) +  \frac{b}{a}\Phi (x,t,p_\infty ) -  \partial _t b  \RP= 0.
    \end{equation} 
   This is the so-called complementarity condition. We prove that $(u_\infty, p_\infty)$ does indeed satisfy \eqref{eq:complementarity} in the sense of distributions. 
   
   The complementarity condition indicates that, for each time $t$, there are two regions of interest: the region where $p_\infty$ is zero, and the region where $p_\infty$ is positive (and therefore the term in the parentheses of \eqref{eq:complementarity} is identically zero). Thus, it is natural to attempt to characterize the evolution of the boundary between these two regions. 
   Examining \eqref{eqn: PME Pressure 2} suggests that, on this boundary, the limit as $m\rightarrow \infty$ of the lefthand side of \eqref{eqn: PME Pressure 2} should be identically zero:
   \[
   \frac{\partial_tp_\infty}{|\nabla p_\infty|}=\frac{|\nabla p_\infty|   }{a}.
   \]
   The lefthand side of the previous line is exactly the    normal velocity of the zero level set of $p_\infty$. Thus, formally, the limiting normal velocity is exactly $\frac{|\nabla p_\infty|   }{a}$. The model \eqref{eqn: PME} is the first one for which an inhomogeneity appears in the $|\nabla p_\infty|$ term of the limiting velocity upon taking the incompressible limit, even in the absence of external density. Making the heuristics about the velocity of $\partial\{p_\infty>0\}$ rigorous is left for future work.

\subsection{Notation and assumptions}
For $T>0$, we denote $Q_T = \R^d \times (0,T)$, and use $C$ to denote any positive constant that is independent of $m$.

\label{ss:assumptions}
\begin{Assump}[Basic assumptions]\label{Assump: a,b,Phi}
    Suppose  there exists $\Lambda>0$ such that 
    \[
    1/\Lambda \leq a(x,t),\, b(x,t) \leq \Lambda
    \]
     for all $(x,t) \in \R^d\times (0,\infty)$ and $\|a\|_{C^4(\R^d\times (0,\infty))}+ \|b\|_{C^4(\R^d\times (0,\infty))}\leq \Lambda$. 
    
    Suppose $\Phi \in C^3 (\R^d\times (0,\infty) \times [0, \infty])$, and, for some $p_M>0$ and $\lambda>0$,  satisfies $\partial _p \Phi(x,t,p) \leq -\lambda$  and $\Phi (x,t,p_M) = 0$, for all $x$, $t$, and $p$.  
\end{Assump}

\begin{Assump}[Initial Data]  \label{Assump: initial data conditions}
For some $u^0 \in L^1 (\R^d)$, suppose that the initial data $u_m ^0$ satisfies, 
\begin{align} 
    \begin{cases}
        u_m ^0 \geq 0,& \quad  \frac{m}{m-1} \LP \frac{u_m ^0}{b} \RP ^{m-1}  \leq p_M, \\
        \Vert u_m ^0 - u^0 \Vert _{L^1 (\R^d)} \rightarrow 0 \ \text{as} \ m \rightarrow \infty,& \quad \Vert \partial _{x_i} u_m^0 \Vert _{L^2 (\R^d)} \leq C, i = 1, \ldots, d,
    \end{cases} 
\end{align}
and $\text{supp}(u_m ^0) \subset \Omega _0$ for some compact $\Omega _0 \subset \R^d$. 
Suppose also that there exists a constant $C>0$ such that 
\begin{align} \label{eq: assump p0 bds}
     \Vert \nabla p_m ^0 \Vert _{L^2 (\R^d)} + \Vert \Delta p_m ^0 \Vert _{L^2 (\R^d)} + \Vert \partial _t p_m^0 \Vert _{L^1 (\R^d)} \leq C.
\end{align}
\end{Assump}

\begin{Assump}[Coefficients] \label{Assump: supersol construction}
    Suppose that either, 
    \begin{enumerate}[(i)]
    \item $d=1$; or,
        \item \label{item:radial} $a(x,t) = a(|x|,t)$, $b(x,t) = b(|x|,t)$ for all $x$, $t$, i.e. $a$, $b$ are radial in space; or, 
        \item \label{item:growth}  there exists $R> 0$ and $0< \e \leq \frac{d - \frac{1}{2}}{\Lambda ^2}$ such that, if $|x| \geq R$, then         \[
        \left \vert \nabla \LP b(x,t)/a(x,t) \RP \right \vert \leq \e|x|^{-1}.
        \] 
    \end{enumerate}
\end{Assump}

\begin{Assump}[Structural]
\label{A:tildelambda}
    Let $\lambda$ be as in Assumption \ref{Assump: a,b,Phi}. Suppose there exists $\tilde\lambda>0$ such that, for all $x,t\in \R^d\times (0,\infty)$,
       \[
    \Delta \log\left(b(x,t)/a(x,t)\right)\geq \tilde{\lambda}-\lambda.
    \]
    \end{Assump}

\begin{Rem}[Remarks on the assumptions]
Assumptions  \ref{Assump: a,b,Phi} and \ref{Assump: initial data conditions}  are analogous to standard assumptions  in the literature, such as in \cites{david2021free,  GKM2022HS, chu2022hele, PQV2014HS}.  

The third assumption is similar to \cite[Assumption 3]{chu2022hele}. Indeed, both  \cite[Assumption 3]{chu2022hele} and our Assumption \ref{Assump: supersol construction} are used to construct barriers to the equation for the pressure in order to conclude that $p_m$ is bounded and has compact support, independent of $m$: see Lemma \ref{lem:basic}(\ref{item:Linfty and cpt supp}). 

Assumption \ref{A:tildelambda} is used in an essential way in establishing a uniform bound on the time derivative of the density and on the time derivative of the pressure in Lemma \ref{lem:basic}(\ref{item:u L1}). We are not aware of an analogous assumption appearing elsewhere in the literature.

Relaxing the requirements of   Assumptions \ref{Assump: supersol construction} and  \ref{A:tildelambda} is a direction for future work. We note that if Lemma \ref{lem:basic}(\ref{item:Linfty and cpt supp}),(\ref{item:u L1}) could be established by other means, the proofs of the other results in this paper would go through without using Assumptions \ref{Assump: supersol construction} and  \ref{A:tildelambda}. 

We note that one choice of $\Phi$ satisfying Assumption \ref{Assump: a,b,Phi} is  $\Phi (x,t, p) = g(p) h(x,t)$, where $g$ is strictly decreasing, $g(p_M) = 0$, and $h$ is a positive function that is bounded. A standard choice for $g$ is a linear function $g(p) = C(p_M - p)$, where $C>0$. 
\end{Rem}

\subsection{Main Results}

Our  main result is that solutions of \eqref{eqn: PME density} converge to those of \eqref{eq:limiting eq} and satisfy the complementarity relation \eqref{eq:complementarity}.

\begin{Th}[Convergence and complementarity relation] \label{Prop: main prop}
Suppose Assumptions \ref{Assump: a,b,Phi}, \ref{Assump: initial data conditions},  \ref{Assump: supersol construction}, and \ref{A:tildelambda}  hold and fix $T>0$. Suppose $u_m$ is a weak solution of \eqref{eqn: PME} on $\R^d\times (0,T)$. Then, up to a sub-sequence,  $u_m$ and $p_m$ converge strongly in $L^1 (Q_T)$ as $m \rightarrow \infty$ to $u_\infty\in L^\infty(Q_T)$, $p_\infty\in L^2(0,T; H^1(\R^d))$ respectively, which satisfy  \eqref{eq:limiting eq} in the weak sense on $\R^d\times (0,T)$. Moreover,
     \eqref{eq:complementarity} holds in the sense of distributions on $\R^d\times (0,T)$.

\end{Th}

\begin{Rem}[Existence and uniqueness for \eqref{eq:limiting eq}] We remark that our main convergence result implies the existence of weak solutions to \eqref{eq:limiting eq}. The comparison principle (and therefore uniqueness) for  solutions to \eqref{eq:limiting eq} follows by a variant of Hilbert's duality method, which is by now standard for equations like \eqref{eq:limiting eq}: see, for example, \cite[Section 3]{PQV2014HS}, \cite[Section 4]{chu2022hele}, and \cite[Section 5]{david2021convective}. The uniqueness of solutions to \eqref{eq:limiting eq}, together with our main result,  thus implies that the entire sequence $(u_m, p_m)$ converges to the unique weak solution $(u_\infty, p_\infty)$ of \eqref{eq:limiting eq}. 
\end{Rem}

\subsection{Strategy of proof}
The main challenge in proving Theorem \ref{Prop: main prop} is obtaining enough uniform (in $m$) estimates on the $u_m$ and the $p_m$. To achieve this, we begin by establish some basic bounds, such as  on the $L^1$ norm of the time and space derivatives of $u_m$ and $p_m$ (Lemma \ref{lem:basic}). Then we prove two key estimates on the pressure. One is a uniform bound on $\|\nabla p_m\|_{L^4(\R^d)}$ (Proposition \ref{Prop: L3 for Dp}). The other is  an estimate of Aronson-B\'enilan type: we bound the negative part of $\Delta p_m$ in $L^3$, uniformly in $m$ (Proposition \ref{Prop: AB-est L3}).

 Note that the righthand side of \eqref{eqn: PME Pressure 2} may be expressed as 
 \[
 (m-1) p_m\left(\frac{1}{a}\Delta p_m +\frac{1}{b}\nabla\left( \frac{b}{a}\right)\cdot\nabla p_m + \frac{1}{a}\Phi (x,t,p_m) -\frac{\partial_tb}{b}\right).
 \]
On the one hand, we note that there is a drift term, $\frac{1}{b}\nabla\left( \frac{b}{a}\right)\cdot \nabla p_m$. We handle the difficulties posed by this term by adapting techniques developed for incompressible limits in the presence of drift and advection  in Chu's \cite{chu2022hele} 
 and David and Schmidtchen's \cite{david2021convective}, respectively.  On the other hand, we see that the highest order  term,  $\Delta p_m$, has a variable coefficient, $\frac{1}{a}$. Our work is the first one to study the incompressible limit in the presence of such a term, which presents new challenges.  
 For example, even when establishing the basic estimates of Lemma \ref{lem:basic}, we have to first bound $\|\partial_t u_m\|_{L^1(\R^d)}$  and then use that  to control  $\|\nabla u_m\|_{L^1(\R^d)}$, whereas in other works, the estimate on the spacial derivative can be established first. And, when establishing the stronger bounds on the pressure, there are many more terms that need to be kept track of and controlled.

\begin{Rem} \label{remark: test function approx}
    We often manipulate the equations for $u_m$ and $p_m$ pointwise and/or differentiate the equations. These manipulations are justified by: (i) first, approximating $u_m$, the solution of \eqref{eqn: PME} with initial data $u^0$, by $u_{m,\epsilon}$, the solution of \eqref{eqn: PME} with initial data $u_m^0+\epsilon$, (ii)  second, establishing the desired estimate for $u_{m,\epsilon}$, which is  uniformly positive and thus smooth, and (iii) third, taking the limit $\epsilon\rightarrow 0$. See, for example, \cite{Vazquez2007PME}*{Section 9.3}.
\end{Rem}

\subsection{Literature review}
\label{ss:lit}

The incompressible limit of a porous medium equation with source term was first studied by Perthame, Quir\`{o}s, and V\'{a}zquez in \cite{PQV2014HS}. Since then, there has been a vast number of works on problems of this type, including \cite{david2021free, DavidPhenotypic, DebiecIncompressibleLimitOneD,DebiecIncompressibleLimitAnyD, degond2022multi, dou2020tumor, gomes2024hele, GKM2022HS, he2023incompressible, he2024incompressible, MPQ, kim2023incompressible, KP, KT, PV}. In particular,  there are several works on the incompressible limit in the presence of drift or advection: \cite{chu2022hele,david2021convective, KPW2019}. Of these, the results of \cite{chu2022hele}, which studies the incompressible ($k\rightarrow \infty$) limit of,
\[
\partial_t\rho_k = \nabla \cdot (\rho_k (\nabla p_k +\overrightarrow{b}))+f\rho_k, \quad p_k = \frac{k}{k-1}\left(\frac{\rho_k}{m}\right)^{k-1},
\]
where  $\overrightarrow{b}(x,t)$, $f(x,t)$ and $m(x,t)$ are given, is closest to ours. Indeed, the coefficient $m(x,t)$ in \cite{chu2022hele} exactly plays the role of the coefficient $b(x,t)$ in our work. However, we do not consider a drift term $\overrightarrow{b}$, and the  PDE studied in \cite{chu2022hele} doesn't include the coefficient $a$. Thus, our result does not imply the results in \cite{chu2022hele}, and vice-versa. We also mention  the   work \cite{turanova2025hele}, which establishes stochastic homogenization for a free boundary problem with heterogeneity both on the interior and on the free boundary, like \eqref{eq:limiting eq}.

\subsection{Outline}
In Section \ref{s:basic estimates}, we establish the basic bounds of Lemma \ref{lem:basic}. 
Section \ref{s:stronger} is devoted to the two stronger estimates on the pressure: a uniform bound on $\|\nabla p_m\|_{L^4(\R^d)}$ (Proposition \ref{Prop: L3 for Dp}) and an estimate of Aronson-B\'enilan type on the negative part of $\Delta p_m$ in $L^3$ (Proposition \ref{Prop: AB-est L3}). 
Finally, in Section  \ref{s:main pf}, we put everything together to prove Theorem \ref{Prop: main prop}.

 \subsection{Acknowledgements}
 Both authors acknowledge support by NSF DMS grant 2204722.

\section{Basic estimates}
\label{s:basic estimates}

In this section we establish basic estimates for solutions of \eqref{eqn: PME}. We denote 
 \begin{equation}
 \label{eq:bgamma}
  \bgamma = \nabla \log \left(b/a\right) \text{ and } \quad \tilde{\Phi}(x,t,p_m) = \Phi(x,t,p_m)-a \partial_t\log b.
   \end{equation}
For some calculations we  use the following equivalent formulation of (\ref{eqn: PME Pressure 2}):
\begin{align} \label{eqn: PME Pressure}
 \partial _t p_m  - 
\frac{\vert \nabla p_m \vert ^2}{a}
    &= (m-1) \frac{p_m}{a}\LP \Delta p_m +  \nabla p_m \cdot \bgamma     + \tilde\Phi (x,t,p_m) \RP.
\end{align}
Furthermore,  we   use the normalized density $v_m := \frac{u_m}{b}$, which solves
\begin{align} \label{eqn: normalized density}
   a\partial _t v_m =  \Delta v^m_m +\bgamma \cdot \nabla v^m_m + v_m \tilde{\Phi}(x,t, p_m).
   \end{align}

\begin{Lemma}[Basic estimates] 
\label{lem:basic}
  Suppose Assumptions \ref{Assump: a,b,Phi}, \ref{Assump: initial data conditions},  \ref{Assump: supersol construction}, and \ref{A:tildelambda}  hold and fix $T>0$. There exits a constant $C>0$, independent of $m$, such that:
  \begin{enumerate}[(i)]
  \item \label{item:Linfty and cpt supp} $0 \leq u_m, p_m, v_m  \leq C$ and $\supp (u_m(\cdot, t))\subset B_{Ct}$ for almost all $0\leq t\leq T$, 
    \item \label{item:L2 bd for Dp} $\Vert \nabla p_m \Vert _{L^1 (Q_T)}+  \Vert \nabla p_m \Vert _{L^2 (Q_T)} \leq C$,
    \item \label{item:u L1} $\Vert \partial_t u_m \Vert _{L^1 (Q_T)}  + \Vert \nabla u_m \Vert _{L^1 (Q_T)}\leq C$,
    \item \label{item:pt}
  $\Vert \partial_t p_m \Vert _{L^1 (Q_T)}\leq C$.
\end{enumerate}
\end{Lemma}

\begin{proof}
We omit writing the subscript $m$ during this proof. Further, we usually abbreviate $\Phi(x,t,p)$ as $\Phi(p)$, and similarly for $\tilde\Phi$.

A small modification of the construction of  \cite{chu2022hele}*{Lemmas 8.1 - 8.3}, namely, replacing the coefficient $m(x,t)$ of \cite{chu2022hele} with $a(x,t)/b(x,t)$, provides a supersolution $Z(x,t)$ of \eqref{eqn: PME Pressure 2}  that satisfies $Z(x,0)\geq p^0(x)$. The comparison principle thus ensures $p(x,t)\leq Z(x,t)$ for all $t>0$. By construction, $Z(x,t)$ is bounded and compactly supported in $x$ for all $t$; therefore, so is $p$. The desired bounds in (\ref{item:Linfty and cpt supp}) on $u$, $v$, and $p$ follow directly.

To establish (\ref{item:L2 bd for Dp}), we begin by integrating \eqref{eqn: PME Pressure} on $Q_T$ to obtain,
\begin{equation}
\label{eq:pt on QT}
\iint_{Q_T} p_t = \iint_{Q_T} \frac{\vert \nabla p \vert ^2}{a}
    + (m-1) \iint_{Q_T} \frac{p}{a} \Delta p +   (m-1)\iint_{Q_T}  \frac{p}{a}   \nabla p \cdot \bgamma     +   (m-1)\iint_{Q_T}  \frac{p}{a}  \tilde\Phi (p).
\end{equation}
Integrating by parts, rearranging the terms, and then using Young's inequality, yields,
\begin{align*}
\iint_{Q_T} \frac{p}{a} \Delta p +   \iint_{Q_T}   \frac{p}{a} \nabla p \cdot \bgamma  & =
-\iint_{Q_T} \nabla\left(\frac{p}{a}\right)\cdot \nabla  p +   \iint_{Q_T}   \frac{p}{a} \nabla p \cdot \bgamma \\
&=
-\iint_{Q_T} \frac{|\nabla p|^2}{a} -  \iint_{Q_T}p\nabla\left(\frac{1}{a}\right)\cdot \nabla p+ \iint_{Q_T}   \frac{p}{a} \nabla p \cdot \bgamma \\
& = -\iint_{Q_T} \frac{|\nabla p|^2}{a} + \iint_{Q_T} \frac{\nabla p}{a}\cdot \left( -p \frac{\nabla a}{a} +p \bgamma\right)\\
&\leq -\iint_{Q_T} \frac{|\nabla p|^2}{a} +\frac{1}{2}  \iint_{Q_T}\frac{|\nabla p|^2}{a} + \frac{1}{2}\iint_{Q_T}\frac{p^2 |-\sfrac{\nabla a}{a} +\bgamma |^2}{a}\\
&\leq  -\frac{1}{2}\iint_{Q_T} \frac{|\nabla p|^2}{a}+C,
\end{align*}
where the final inequality follows from  item (\ref{item:Linfty and cpt supp}) of this lemma and Assumption \ref{Assump: a,b,Phi}. Now, using this estimate to bound the righthand side of \eqref{eq:pt on QT} from above yields,
\[
\int_{\R^2} p(x,T)\, dt - \int_{\R^2} p(x,0)\, dt \leq  \iint_{Q_T} \frac{\vert \nabla p \vert ^2}{a}
    -  \frac{m-1}{2}\iint_{Q_T} \frac{|\nabla p|^2}{a}+(m-1)C   +   (m-1)\iint_{Q_T}  \frac{p}{a}\tilde\Phi (p),
\]
which upon rearranging becomes,
\begin{align*}
 \frac{m-3}{2}\iint_{Q_T} \frac{|\nabla p|^2}{a}&\leq (m-1)C +   (m-1)\iint_{Q_T}  \frac{p}{a}\tilde\Phi (p) - \int_{\R^2} p(x,0)\, dt\\
 &\leq (m-1)C+C,
 \end{align*}
where the final inequality follows from item (\ref{item:Linfty and cpt supp}) of this lemma, Assumption \ref{Assump: a,b,Phi}, and Assumption \ref{Assump: initial data conditions}. Dividing by $(m-3)$ and using that $a$ is uniformly bounded from above yields the desired estimate on $\|\nabla p\|_{L^2(\R^d)}$. The compact support of $p$ along with H\"older's inequality then gives the desired uniform $L^1$ bound on $\nabla p$ as well.

Before proceeding, we note that the definition of $v$ and  item (\ref{item:Linfty and cpt supp}) imply,
    \[ 
    |\nabla v_m^m | =  v_m |\nabla p_m | \leq C|\nabla p_m |,
    \]
    which, upon integrating over $Q_T$ and applying item (\ref{item:L2 bd for Dp}), yields
\begin{equation}
\label{eq:L1 L2 bound on v_m}     
\Vert \nabla v_m ^m\Vert _{L^1 (Q_T)}
    + \Vert \nabla v_m ^m\Vert _{L^2 (Q_T)} \leq C.
    \end{equation}

Next we will show that the space and time derivatives of the normalized density are bounded in $L^1$, and deduce (\ref{item:u L1}) from there.  The argument relies on  the fact that $\Phi$, and therefore $\tilde{\Phi}$, is strictly decreasing in $p$, as well as on Assumption \ref{A:tildelambda}. 

We differentiate \eqref{eqn: normalized density} in $t$ to obtain,
\begin{align*}
(av_t)_t &= \Delta(v^m)_t + \bgamma \cdot \nabla (v^m)_t + \bgamma_t\cdot \nabla v^m + v_t \tilde\Phi +v\left(\tPhi_t(p)+\tPhi_p(p)p_t\right)\\
&= \Delta(mv^{m-1}v_t) +\nabla \cdot (\bgamma m v^{m-1}v_t) - (\nabla\cdot \bgamma)(v p_t) +\bgamma_t\cdot \nabla v^m + v_t \tilde\Phi +v\left(\tilde\Phi_t(p)+\tilde\Phi_p(p)p_t\right),
\end{align*}
where we've used the equality $(v^m)_t = v p_t$. Next we multiply by $\sgn(v_t)=\sgn(p_t)$ and apply Kato's inequality to find,
\begin{align*}
(a|v_t|)_t &\leq \Delta(mv^{m-1}|v_t|) +\nabla \cdot (\bgamma m v^{m-1}|v_t|) - (\nabla\cdot \bgamma)(v |p_t|) +
\\&\quad +\bgamma_t\cdot \nabla v^m\sgn(v_t) + |v_t |\tilde\Phi +v\tilde\Phi_t(p)\sgn(v_t)+v\tilde\Phi_p(p)|p_t|.
\end{align*}
It is at this point that we use Assumption \ref{A:tildelambda} to  show that the sum of the two terms involving $v|p_t|$ is nonpositive. Indeed, upon also recalling the definition of $\bgamma$, we have, 
\begin{align*}
-(\nabla\cdot \bgamma)(v |p_t|)+v\tilde\Phi_p(p)|p_t| & = \left(-\Delta \log\left(b/a\right)+\tilde\Phi_p(p)\right) v|p_t|
\\ &\leq \left(-\Delta \log\left(b/a\right)-\lambda\right)v|p_t| \leq -\tilde\lambda v|p_t|,
\end{align*}
where we used Assumption \ref{Assump: a,b,Phi} to obtain the first inequality, and Assumption \ref{A:tildelambda} to obtain the second. Now we use this in the estimate on $(a|v_t|)_t$ and integrate over $Q_T$. Since $v$ is compactly supported, the first two term vanish. Thus we find,
\begin{align}
 \nonumber
\frac{d}{dt}\iint_{Q_T}a|v_t| &\leq \iint_{Q_T}\bgamma_t\cdot \nabla v^m\sgn(v_t) + \iint_{Q_T}|v_t |\tilde\Phi +\iint_{Q_T}v\tilde\Phi_t(p)\sgn(v_t)  -\tilde\lambda\iint_{Q_T} v|p_t|
\\&\label{eq:ddt vt} \leq C+ C \iint_{Q_T}a|v_t |  -\tilde\lambda\iint_{Q_T} v|p_t|
\\&\leq C+ C \iint_{Q_T}a|v_t | , \nonumber
\end{align}
where to obtain the second inequality we've used \eqref{eq:L1 L2 bound on v_m}, item (\ref{item:Linfty and cpt supp}), and Assumption \ref{Assump: a,b,Phi}.  Gronwall's inequality and Assumption  \ref{Assump: initial data conditions} therefore implies,
\begin{equation}
\label{eq:bd vt}
 \iint_{Q_T}a|v_t| \leq C.
\end{equation}

We will now use \eqref{eq:bd vt} to bound the $L^1$ norm of the spacial derivative of $v$. 
To this end, we  differentiate \eqref{eqn: normalized density} in $x_i$, for some $i\in \{1,.., d\}$, to obtain (after some standard manipulations),
\begin{align*}
\partial_t (av_{x_i}) &=  a \partial_t v_{x_i}+ \partial_t a v_{x_i} = (a \partial_t v)_{x_i}- a_{x_i}\partial_t  v+ \partial_t a v_{x_i}\\
&=\Delta (v^m)_{x_i} +\bgamma \cdot \nabla (v^m)_{x_i} + \bgamma_{x_i}\cdot\nabla v^m + v_{x_i}\tilde\Phi(p) + v \tilde\Phi(p)_{x_i}- a_{x_i}\partial_t  v+ \partial_t a v_{x_i}.
\end{align*}
Performing further standard manipulations on the first two terms yields,
\begin{align*}
\partial_t (av_{x_i}) &=\Delta (m v^{m-1} v_{x_i}) + \nabla\cdot (\bgamma m v^{m-1} v_{x_i}) - (\nabla \cdot \bgamma)(v^m)_{x_i} +\\ &\quad + \bgamma_{x_i}\cdot\nabla v^m +  v_{x_i}\tilde\Phi(p) + v \tilde\Phi(p)_{x_i}- a_{x_i}\partial_t  v+ \partial_t a v_{x_i}.
\end{align*}
Next, we multiply by $\sgn(v_{x_i})$ and use Kato's inequality  to find,
\begin{align*}
\partial_t (a|v_{x_i}|) &\leq \Delta (m v^{m-1} |v_{x_i}|) + \nabla\cdot (\bgamma m v^{m-1} |v_{x_i}|) - (\nabla \cdot \bgamma)(v^m)_{x_i}\sgn(v_{x_i}) +\\
&\quad+ \bgamma_{x_i}\cdot \nabla v^m\sgn(v_{x_i}) +  |v_{x_i}|\tilde\Phi(p) + v \tilde\Phi(p)_{x_i}\sgn(v_{x_i})- a_{x_i}\partial_t  v\ \sgn(v_{x_i})+ \partial_t a |v_{x_i}|.
\end{align*}
Now we integrate over $Q_T$. Note that, since $v$ is compactly supported, the first two terms on the righthand side disappear, so we obtain, 
\begin{align*}
\iint_{Q_T}\partial_t (a|v_{x_i}|) &\leq \iint_{Q_T} - (\nabla \cdot \bgamma)(v^m)_{x_i}\sgn(v_{x_i}) + \iint_{Q_T}\bgamma_{x_i}\cdot \nabla v^m\sgn(v_{x_i}) +\iint_{Q_T}  |v_{x_i}|\tilde\Phi(p)  \\&\quad +\iint_{Q_T}   v \tilde\Phi(p)_{x_i}\sgn(v_{x_i})- \iint_{Q_T}a_{x_i}\partial_t  v\ \sgn(v_{x_i})+ \iint_{Q_T}\partial_t a |v_{x_i}|.
\end{align*}
Our assumptions on the coefficients, together with the estimates \eqref{eq:L1 L2 bound on v_m}  and \eqref{eq:bd vt}, imply that the first two terms, as well as the fourth term, on the righthand side are bounded from above, uniformly in $m$. Thus we have,
\begin{align*}
\frac{d}{dt}\iint_{Q_T}a|v_{x_i}| &\leq C+ C\iint_{Q_T}  a|v_{x_i}|+ C\iint_{Q_T}|\partial_t  v|\leq C+ C\iint_{Q_T}  a|v_{x_i}|,
\end{align*}
where to obtain the first inequality we've  used Assumption \ref{Assump: a,b,Phi}, and we've used \eqref{eq:bd vt} to obtain the second inequality. Applying Gronwall's inequality and using Assumption \ref{Assump: initial data conditions} yields 
\[
\iint_{Q_T}a|v_{x_i}| \leq C.
\]
The  estimates in (\ref{item:u L1})  follow directly from the previous line, the estimate \eqref{eq:bd vt}, the definition of $v$, and Assumption \ref{Assump: a,b,Phi}.

Finally, we move on to the proof of (\ref{item:pt}). We note that the inequality \eqref{eq:ddt vt} implies,
\begin{align*}
\int_{\R^d}a|v_t(x,T)|\, dx - \int_{\R^d}a|v_t(x,0)|\, dx + \tilde\lambda\iint_{Q_T} v|p_t| \leq & C+ C \iint_{Q_T}a|v_t |.
\end{align*}
Rearranging, and then using  (\ref{item:u L1}) and Assumption \ref{Assump: initial data conditions} yields,
\begin{equation}
\label{eq:vpt}
 \tilde\lambda\iint_{Q_T} v|p_t|\leq  C+ C \iint_{Q_T}a|v_t |+     \int_{\R^d}a|v_t(x,0)|\, dx\leq C.
\end{equation}
The definitions of $v$ and $p$ imply $p_t = m v^{m-2} v_t$; therefore,
\begin{align*}
\iint_{Q_T} |p_t| &= \iint_{Q_T\cap \{v<1/2\}} mv^{m-2}|v_t| + \iint_{Q_T\cap \{v\geq 1/2\}}  |p_t|\\
&\leq C\iint_{Q_T\cap \{v<1/2\}} |v_t| + 2\iint_{Q_T\cap \{v\geq 1/2\}} v |p_t|.
\end{align*}
The desired estimate thus follows directly from \eqref{eq:bd vt} and \eqref{eq:vpt}, completing the proof.

\end{proof}

 \section{Stronger estimates on the pressure}
\label{s:stronger}

Throughout this section, we denote  $\Omega_T=B_{CT} \times (0,T)$, where $C>0$ is large enough so that the support of $p(\cdot, t)$ is contained in $B_{CT}$ for almost every  $t \in [0,T]$. The existence of such a $C$  is guaranteed by  Lemma \ref{lem:basic}(\ref{item:Linfty and cpt supp}).

\subsection{$L^3$ bound for the gradient of the pressure}
 Let us denote $\bar\Phi(x,t,p_m)= \frac{b}{a}\Phi (x,t,p_m) -\partial_tb$ and define
\[
\omega_m:= \nabla \cdot\left(\frac{b}{a}\nabla p_m\right) 
    + \bar\Phi(p_m).
\]
Notice that  the equation \eqref{eqn: PME Pressure 2} satisfied by $p_m$ may now be written as,
\begin{equation}
\label{eq:pomega}
\partial _t p_m  - 
    \frac{\vert \nabla p_m \vert ^2}{a} 
 =(m-1) \frac{p_m}{b}\omega_m.
\end{equation}
\begin{Prop}[$L^4$ bound for $\nabla p_m$] \label{Prop: L3 for Dp}
Suppose Assumptions \ref{Assump: a,b,Phi}, \ref{Assump: initial data conditions},  \ref{Assump: supersol construction}, and \ref{A:tildelambda}  hold and fix $T>0$. Let $T>0$. There exist positive constants $C_1$,  $C_2$, and $C_3$, independent of $m$, such that
\begin{align}
\label{eq:PropDp}
\frac13\iint_{\Omega_T} p_m\frac{b}{a^2}\sum_{i,j=1}^d \left|\frac{\partial^2p_{m}}{\partial x_i\partial x_j}\right|^2+  (m-C_1) \iint_{\Omega_T}\frac{p_m}{b}\omega_m^2&\leq C_2, \text{ and }\\
    \iint _{\Omega _T} \vert \nabla p_m \vert ^4 &\leq C_3. \label{eq:L4bd}
\end{align}
\end{Prop}
Analogous bounds were established in  \cite[Lemma 5.3] {chu2022hele} (in the presence drift, and with $a\equiv 1$) and in   \cite[Lemma 3.2]{david2021convective} (in the presence of advection, and with $a\equiv b\equiv 1$).  
The proof of Proposition \ref{Prop: L3 for Dp} is inspired by the proofs of these two results.

We will employ the following lemma. Its proof exactly follows that of \cite[Lemma 5.1]{chu2022hele} in the special case of, in the notation of \cite{chu2022hele}, $m\equiv 1$. The same argument can also be found at the conclusion of the proof of Lemma 3.2 in \cite{david2021convective}. Thus, we omit writing the proof of the lemma here.

\begin{Lemma}
\label{lem:chu}
For any $\phi\in C^\infty_c (\R^d)$, 
 \begin{align*}
\int_{\R^d}|\nabla \phi |^4& \leq
 2\int_{\R^d} |\phi|^2|\Delta \phi|^2 + 8\sum_{i,j=1}^d \int_{\R^d} |\phi|^2 |\phi_{x_i, x_j}|^2.
 \end{align*}
\end{Lemma}

For the remainder of this subsection, we will  drop the subscript $m$, and unless otherwise noted, all integrals are taken over $\Omega_T$.

The proof of Proposition \ref{Prop: L3 for Dp} proceeds by multiplying the  equation \eqref{eqn: PME Pressure 2} satisfied by $p$ by $\omega$ and performing a variety of manipulations. For the sake of organization, we isolate a particular estimate into the following lemma. 
\begin{Lemma}
\label{lem:L3 bd}
Under the assumptions of Proposition \ref{Prop: L3 for Dp}, we have
\begin{align}
\label{eq:lem L3}
\iint \frac{b}{a^2}|\nabla p|^2 \Delta p  \geq & - C\iint p \frac{b}{a^2}|\Delta p|^2 + \frac23\iint p\frac{b}{a^2}\sum_{i,j=1}^d |p_{x_i,x_j}|^2-\frac43\iint \nabla p \cdot \nabla\left(\frac{b}{a^2}\right)|\nabla p|^2-C.
\end{align}
\end{Lemma}
\begin{proof}[Proof of Lemma \ref{lem:L3 bd}]
Let us denote $I=\iint\frac{b}{a^2}|\nabla p|^2 \Delta p $. Integrating by parts twice yields, 
\begin{align*}
I&= \iint p\Delta \left(\frac{b}{a^2} |\nabla p|^2\right)  = \iint p\left(\frac{b}{a^2}\Delta |\nabla p|^2 + 2\nabla \left(\frac{b}{a^2}\right) \nabla (|\nabla p|^2) + \Delta \left(\frac{b}{a^2}\right)|\nabla p|^2\right).
\end{align*}
Upon integrating by parts in the second term on the right-hand side, we find,
\[
I  = \iint p\frac{b}{a^2}\Delta |\nabla p|^2 - 2 \iint \nabla p\cdot \nabla \left(\frac{b}{a^2}\right) |\nabla p|^2 -2 \iint p \Delta \frac{b}{a^2}|\nabla p|^2 +\iint p \Delta \left(\frac{b}{a^2}\right)|\nabla p|^2.
\]
Applying Lemma \ref{lem:basic}(\ref{item:Linfty and cpt supp})(\ref{item:L2 bd for Dp}) to the last two terms on the righthand side of the previous line, and recalling Assumption \ref{Assump: a,b,Phi}, gives,
\begin{align*}
I&\geq  \iint p\frac{b}{a^2}\Delta |\nabla p|^2- 2 \iint \nabla p\cdot \nabla \left(\frac{b}{a^2}\right) |\nabla p|^2 -C. 
\end{align*}
Now, we first use the identity $\Delta(|\nabla p|^2) = 2\nabla p\cdot \nabla \Delta p + 2\sum_{i,j=1}^d |p_{x_i,x_j}|^2$ and then integrate by parts in the first term to find,
\begin{align*}
\frac12I&\geq  \iint p\frac{b}{a^2}\nabla p\cdot \nabla \Delta p +  \iint p\frac{b}{a^2}\sum_{i,j=1}^d |p_{x_i,x_j}|^2- \iint \nabla p \cdot \nabla\left(\frac{b}{a^2}\right)|\nabla p|^2-C\\
&\geq -\iint  \nabla \cdot \left(p\frac{b}{a^2}\nabla p\right) \Delta p +  \iint p\frac{b}{a^2}\sum_{i,j=1}^d |p_{x_i,x_j}|^2- \iint \nabla p \cdot \nabla\left(\frac{b}{a^2}\right)|\nabla p|^2-C\\
& \geq  -\iint  \frac{b}{a^2}|\nabla p|^2 \Delta p - \iint p \Delta p \nabla \left(\frac{b}{a^2}\right)\cdot \nabla p - \iint p \frac{b}{a^2}|\Delta p|^2 + \iint p\frac{b}{a^2}\sum_{i,j=1}^d |p_{x_i,x_j}|^2-\\
&\quad \quad - \iint \nabla p \cdot \nabla\left(\frac{b}{a^2}\right)|\nabla p|^2-C.
\end{align*} 
We recognize the first term on the righthand side as $-I$, so that, upon rearranging, we find, 
\begin{equation}
\label{eq:I est}
 I \geq -\frac23 \iint p \Delta p \nabla \left(\frac{b}{a^2}\right)\cdot \nabla p - \frac23\iint p \frac{b}{a^2}|\Delta p|^2 + \frac23\iint p\frac{b}{a^2}\sum_{i,j=1}^d |p_{x_i,x_j}|^2-\frac43 \iint \nabla p \cdot \nabla\left(\frac{b}{a^2}\right)|\nabla p|^2-C.
\end{equation}
We will now demonstrate that the first term on the righthand side is bounded below by a multiple (independent of $m$) of the second term, up to a constant. To this end,  we use Young's inequality, followed by  Assumption \ref{Assump: a,b,Phi} and Lemma \ref{lem:basic}(\ref{item:Linfty and cpt supp})(\ref{item:L2 bd for Dp}), and obtain,
\begin{align*}
-\frac23 \iint p \Delta p \nabla \left(\frac{b}{a^2}\right)\cdot \nabla p &\geq
-\frac{1}{3} \iint p |\Delta p|^2 - \frac13\iint p  \left|\nabla \frac{b}{a^2}\right||\nabla p|^2\\
&\geq - C\iint p \frac{b}{a^2}|\Delta p|^2 -C.
\end{align*}
Using this to bound the righthand side of \eqref{eq:I est} from below yields the desired inequality \eqref{eq:lem L3}.

\end{proof}
We proceed with:

\begin{proof}[Proof of Proposition \ref{Prop: L3 for Dp}]
Multiplying \eqref{eq:pomega} by $\omega$ and integrating over $\Omega_T$ yields,
\begin{equation}
\label{eq:mult omega}
\iint p_t \omega - 
\iint\frac{\vert \nabla p \vert ^2}{a}\nabla \cdot\left(\frac{b}{a}\nabla p\right) - \iint\frac{\vert \nabla p \vert ^2}{a}\bar\Phi(p)
    = (m-1) \iint\frac{p}{b}\omega^2.
\end{equation}
 Using the definition of $\omega$; followed by an integration by parts, Lemma \ref{lem:basic}, and Assumption \ref{Assump: a,b,Phi}; and then the product rule, yields,
\begin{align*}
\iint p_t \omega&\leq \iint p_t \nabla \cdot\left(\frac{b}{a}\nabla p\right) 
    + \iint p_t \bar\Phi(p)=-\iint \nabla p_t \cdot\left(\frac{b}{a}\nabla p\right) +C\\
    &\leq-\frac{1}{2}\iint \frac{b}{a}|\nabla p|^2_t  +C=-\frac{1}{2}\iint \left(\frac{b}{a}|\nabla p|^2\right)_t + \frac{1}{2}\iint \left(\frac{b}{a}\right)_t|\nabla p|^2+C\\
    &\leq-\frac{1}{2}\iint \left(\frac{b}{a}|\nabla p|^2\right)_t +C,
\end{align*}
where the final equality follows from Assumption \ref{Assump: a,b,Phi}
 and the bounds  of Lemma \ref{lem:basic}. This inequality,  followed by \eqref{eq:mult omega}, give, 
 \[
 \frac{1}{2}\iint \left(\frac{b}{a}|\nabla p|^2\right)_t  \leq -\iint p_t \omega +C = -(m-1)\iint\frac{p}{b}\omega^2 -\iint\frac{\vert \nabla p \vert ^2}{a}\nabla \cdot\left(\frac{b}{a}\nabla p\right) - \iint\frac{\vert \nabla p \vert ^2}{a}\bar\Phi(p)+C.
 \]
 Assumption \ref{Assump: a,b,Phi}
 and Lemma \ref{lem:basic}  imply $\left|\iint\frac{\vert \nabla p \vert ^2}{a}\bar\Phi(p)\right|\leq C$. Thus we obtain,
\begin{equation}
\label{eq:mult omega 2}
\frac{1}{2}\frac{d}{dt}\iint \frac{b}{a}|\nabla p|^2 +   (m-1) \iint\frac{p}{b}\omega^2
    = \underbrace{-\iint\frac{\vert \nabla p \vert ^2}{a}\nabla \cdot\left(\frac{b}{a}\nabla p\right)}_{J} +C.
\end{equation}
We will now demonstrate that  $J$ can be bounded from above by terms that are independent of $m$, plus ones that can be ``absorbed" into $(m-1) \iint\frac{p}{b}\omega^2$. 
We have, 
\begin{align}
\nonumber
J &= -\iint \frac{b}{a^2}|\nabla p|^2 \Delta p  - \iint\frac{\vert \nabla p \vert ^2}{a}\nabla \left(\frac{b}{a}\right)\cdot\nabla p\\
&\leq C\iint p \frac{b}{a^2}|\Delta p|^2 - \frac23\iint p\frac{b}{a^2}\sum_{i,j=1}^d |p_{x_i,x_j}|^2+ \frac43\iint \nabla p \cdot \nabla\left(\frac{b}{a^2}\right)|\nabla p|^2 + C - \nonumber\\
&\quad \quad - \iint\frac{\vert \nabla p \vert ^2}{a}\nabla \left(\frac{b}{a}\right)\cdot\nabla p, \label{eq:J}
\end{align}
where the inequality follows from Lemma \ref{lem:L3 bd}. We will now
estimate the last term on the righthand side of the previous line.  Young's inequality (with $\ep>0$ to be determined) yields,
\begin{equation}
\label{eq:J2}
\left|\iint\frac{\vert \nabla p \vert ^2}{a}\nabla \left(\frac{b}{a}\right)\cdot\nabla p\right| \leq 
\ep\iint|\nabla p|^4 +\frac{1}{4\ep}\iint \left| \frac{1}{a}\nabla \left(\frac{b}{a}\right)\cdot\nabla p\right|^2\leq \frac{\ep}{2}\iint|\nabla p|^4 +\frac{C}{\ep},
\end{equation}
where the second inequality follows from Assumption \ref{Assump: a,b,Phi}
 and Lemma \ref{lem:basic}(\ref{item:L2 bd for Dp}).  Similarly we find that the absolute value of the third term on the right-hand side of \eqref{eq:J} is also bounded from above by $\frac{\ep}{2}\iint|\nabla p|^4 +\frac{C}{\ep}$. Thus we find,
 \begin{equation}
 \label{eq:J3}
 J\leq C\iint p \frac{b}{a^2}|\Delta p|^2 - \frac23\iint p\frac{b}{a^2}\sum_{i,j=1}^d |p_{x_i,x_j}|^2 +C+ \ep\iint|\nabla p|^4 +\frac{C}{\ep} .
 \end{equation}
Lemma \ref{lem:chu}, Lemma \ref{lem:basic}(\ref{item:Linfty and cpt supp})(\ref{item:L2 bd for Dp}), and  Assumption \ref{Assump: a,b,Phi}  yield, for some  constant $\tilde{C}>0$, independent of $m$,
 \begin{align*}
\iint|\nabla p |^4& \leq
 C\iint p\frac{b}{a^2}|\Delta p|^2 + \tilde{C}\sum_{i,j=1}^d \iint p\frac{b}{a^2}|p_{x_i, x_j}|^2.
 \end{align*}
 We now  use the previous line  to bound the righthand side of \eqref{eq:J3} and obtain,
 \begin{align}
\nonumber
J &\leq C\iint p \frac{b}{a^2}|\Delta p|^2 - \frac23\iint p\frac{b}{a^2}\sum_{i,j=1}^d |p_{x_i,x_j}|^2+C+\\
&\quad\quad +\ep\left(C +C\iint p\frac{b}{a^2}|\Delta p|^2 + \tilde{C}\sum_{i,j=1}^d \iint  p\frac{b}{a^2}|p_{x_i, x_j}|^2\right) +\frac{C}{\ep}.
\end{align}
Taking $\ep=\min\left\{1,\frac{1}{3\tilde{C}}\right\}$ and combining like terms yields,
\[
J\leq C\iint p \frac{b}{a^2}|\Delta p|^2 - \frac13\iint p\frac{b}{a^2}\sum_{i,j=1}^d |p_{x_i,x_j}|^2+C.
\]
Upon using this estimate to bound the righthand side of  \eqref{eq:mult omega 2} from above, we find,
\begin{equation}
\label{eq:afterJ}
\frac{1}{2}\frac{d}{dt}\iint \frac{b}{a}|\nabla p|^2 + \frac13\iint p\frac{b}{a^2}\sum_{i,j=1}^d |p_{x_i,x_j}|^2+  (m-1) \iint\frac{p}{b}\omega^2
    \leq C\iint p \frac{b}{a^2}|\Delta p|^2 +C.
\end{equation}
To control the first term on the righthand side of the previous line, we recall the definition,
\[
\omega =  \frac{b}{a}\Delta p +\nabla \frac{b}{a}\cdot\nabla p 
    + \bar\Phi(p).
    \]
 Thus,    Lemma \ref{lem:basic}, Assumption \ref{Assump: a,b,Phi}, and a number of standard manipulations imply,
\begin{equation}
\label{eq:WTS}
\iint p \frac{b}{a^2}|\Delta p|^2\leq C \iint\frac{p}{b}\omega^2 +C;
\end{equation}
together with \eqref{eq:afterJ}, this  gives,
\[
\frac{1}{2}\int_{B_{CT}} \frac{b}{a}|\nabla p(x,T)|^2\, dx- \frac{1}{2}\int_{B_{CT}} \frac{b}{a}|\nabla p(x,0)|^2\, dx+ \frac13\iint p\frac{b}{a^2}\sum_{i,j=1}^d |p_{x_i,x_j}|^2+  (m-C) \iint\frac{p}{b}\omega^2\leq C,
\]
from which the desired estimate \eqref{eq:PropDp} follows by Assumption \ref{Assump: initial data conditions}.
Finally, Lemma \ref{lem:chu},  the inequality $|\Delta p|^2\leq d \sum_{i,j=1}^d |p_{x_i,x_j}|^2$, and Lemma \ref{lem:basic}(\ref{item:Linfty and cpt supp}) imply 
\[
\iint |\nabla p|^4\leq C\iint \sum_{i,j=1}^d p^2 |p_{x_i,x_j}|^2.
\]
The estimate \eqref{eq:L4bd} therefore follows  from the previous line,  \eqref{eq:PropDp}, and Assumption \ref{Assump: a,b,Phi}.
\end{proof}

\subsection{$L^3$ Aronson-B{\'e}nilan Estimate} 

We recall the notation $\bgamma$ and $\tPhi$ of \eqref{eq:bgamma},   abbreviate $\tPhi (x,t,p_m)$ as $\tPhi(p_m)$ as before, and 
define 
 \[
 w_m = \frac{a}{b} \nabla \cdot \left(\frac{b}{a} \nabla p_m\right) + \tPhi(p_m).
 \]
  The definitions of $w_m$  and $\bgamma$ yield,
 \begin{equation}
 \label{eq:w Del p}
 w_m= \Delta p_m + \bgamma \cdot \nabla p_m+\tPhi(p_m),
 \end{equation}
  so that the equation  \eqref{eqn: PME Pressure} that $p_m$ satisfies becomes,
 \begin{equation}
 \label{eq:pw}
 \partial_tp_m -\frac{|\nabla p_m|^2}{a}= (m-1)\left(\frac{p_mw_m}{a}\right).
 \end{equation}
We will obtain a bound on the $L^3$ norm of the negative part of $w_m$, where the negative part of any $\alpha\in \R$ is denoted by,
 \[
|\alpha|_-=\begin{cases}
0 &\text{ for } \alpha>0,\\
-\alpha &\text{ for }\alpha\leq 0.
\end{cases}
\]

\begin{Prop}[$L^3$ AB-Estimate] \label{Prop: AB-est L3}
 Suppose Assumptions \ref{Assump: a,b,Phi}, \ref{Assump: initial data conditions},  \ref{Assump: supersol construction}, and \ref{A:tildelambda}  hold. Fix $T>0$.  
 There  exists a positive constant $C$, independent of $m$, such that for
     \begin{equation}
     \label{eq:AB m}
m>\max\left\{2, 1+\frac{(\sup a^{-1} +1)}{\inf a^{-1}}\frac{8}{3}\right\},
\end{equation}
we have
     \begin{align}
     \iint _{\Omega _T} |w_m(x,t)|_-^3\, dx\, dt &\leq C  \label{eq:prop wm3}\quad \text{ and }\quad  \iint _{\Omega _T} \vert \Delta p_m(x,t) \vert \, dx\, dt \leq C.           \end{align}
\end{Prop}

The main idea of the proof of Proposition \ref{Prop: AB-est L3} is similar to that of \cite[Lemma 3.3]{david2021convective}; namely, to use the equation \eqref{eq:pw} to understand the time evolution of $|w_m|_-$. 
The first step is to find a PDE that $w_m$ is a supersolution of. We separate this into a lemma. For the remainder of this subsection, we  drop the subscript $m$.

 \begin{Lemma}[Equation for $|w|_-$]
 \label{lem:AB}
 Under the assumptions of Proposition \ref{Prop: AB-est L3}, we have, 
 \begin{align}
\label{eq:w final}
w_t&\geq \tilde{C}w^2 - C|w| +2a^{-1}\nabla p\cdot \nabla w   -C\left(1+  |\nabla p|^2  +|p_t| \right)+(m-1)\left( \bgamma\cdot\nabla (a^{-1}pw) +\Delta \left(a^{-1}pw\right)\right),
\end{align}
where $\tilde{C}=\frac{3}{8d\sup a}$ and $C_0$ denotes a positive constant independent of $m$.
 \end{Lemma}

 \begin{proof}[Proof of Lemma \ref{lem:AB}]
 Taking the time derivative of \eqref{eq:w Del p} yields,
\begin{equation}
\label{eq:wt}
w_t = \Delta p_t + \bgamma\cdot \nabla p_t + \bgamma_t\cdot \nabla p + (\tPhi(p))_t.
\end{equation}
We will now manipulate the righthand side and replace  derivatives of $p$ of order 2 or higher with $w$ and its derivatives. 
We begin with the first term. We take Laplacian of \eqref{eq:pw} and then perform standard manipulations to obtain,
\begin{align}
\Delta p_t &= \Delta  \LP \frac{\vert \nabla p \vert ^2}{a} \RP + (m-1)\Delta \left(a^{-1}pw\right) \nonumber \\
&= a^{-1}\Delta |\nabla p|^2 + 2 \nabla |\nabla p|^2\cdot \nabla a^{-1} +|\nabla p|^2\Delta a^{-1} + (m-1)\Delta \left(a^{-1}pw\right). \label{eq:Delpt}
\end{align}
We find, upon performing a direct calculation and then using  \eqref{eq:w Del p},
\begin{align}
a^{-1}\Delta \vert \nabla p \vert ^2 &= 2a^{-1}\sum_{i,j=1}^d |p_{x_i,x_j} | ^2 + 2a^{-1}\nabla p\cdot \nabla \Delta p \nonumber\\
&= 2a^{-1}\sum_{i,j=1}^d |p_{x_i,x_j} | ^2 + 2a^{-1}(\nabla p\cdot \nabla w - \nabla p\cdot \nabla(\bgamma \cdot \nabla p) - \nabla p\cdot \nabla\tPhi(p)). \label{eq: Del nabla p 2}
\end{align}
To bound the third term on the righthand side of the previous line, we perform standard calculations, followed by two applications of Young's inequality, and then use Assumption \ref{Assump: a,b,Phi}, to obtain,
\begin{align*}
 \nabla p\cdot \nabla(\bgamma \cdot \nabla p)&=\sum_{i=1}^d p_{x_i}\partial_{x_i}\sum_{j=1}^d \bgamma_j p_{x_j} = \sum_{i,j=1}^d (p_{x_i}p_{x_i,x_j}\bgamma_j  +p_{x_i}\partial_{x_i}\bgamma_j p_{x_j})\leq \\
 &\leq \sum_{i,j=1}^d \left(\frac{|p_{x_i,x_j}|^2}{8} +2\bgamma_j^2p_{x_i}^2  +\frac{(p_{x_i}\partial_{x_i}\bgamma_j )^2}{2}+\frac{p_{x_j}^2}{2}\right)
 \leq \sum_{i,j=1}^d \frac{|p_{x_i,x_j}|^2}{8} +C|\nabla p|^2.
\end{align*}
(Here $\bgamma_j$ denotes the $j$-th component of $\bgamma$.) 
For the fourth term on the righthand side of \eqref{eq: Del nabla p 2}, we find, by using Assumption \eqref{Assump: a,b,Phi},
\[
\nabla p \cdot \nabla \tPhi(p) = \tPhi_p(p)|\nabla p|^2\leq 0.
\]
Using the two previous estimates in \eqref{eq: Del nabla p 2} yields,
 \begin{equation}
\label{eq: Del nabla p 2 2}
a^{-1}\Delta \vert \nabla p \vert ^2 \geq \frac{3}{2 } a^{-1}\sum_{i,j=1}^d |p_{x_i,x_j}| ^2 +2a^{-1}\nabla p\cdot  \nabla w  - C |\nabla p|^2.
\end{equation}
For the second term on the righthand side of \eqref{eq:Delpt}, a direct calculation and Young's inequality yield, 
\[
    2 \nabla \vert \nabla p \vert ^2 \cdot \nabla a^{-1} = -4\sum_{i,j=1}^d p_{x_i, x_j} p_{x_j}\frac{a_{x_i}}{a^2}
 \geq  -\frac{1}{2a}\sum_{i,j=1}^d |p_{x_i, x_j}| ^2 - 8 \frac{\vert \nabla a \vert ^2}{a^3} \vert \nabla p \vert ^2. 
    \]
Using the previous line, \eqref{eq: Del nabla p 2 2}, and  Assumption \ref{Assump: a,b,Phi} to bound the righthand side of  \eqref{eq:Delpt} from below gives,
\begin{equation}
\label{eq:bd Del pt}
\Delta p_t \geq a^{-1}\sum_{i,j=1}^d (p_{x_i,x_j} ) ^2 +2a^{-1}\nabla p\cdot  \nabla w  - C |\nabla p|^2  + (m-1)\Delta \left(a^{-1}pw\right).
\end{equation}

We move on to the second term on the righthand side of \eqref{eq:wt}. Upon using \eqref{eq:pw} and performing the usual calculations, we find,
\begin{align*}
\bgamma\cdot \nabla p_t &= \bgamma\cdot \nabla (|\nabla p|^2 a^{-1}) + (m-1)\bgamma\cdot\nabla (a^{-1}pw)\\
&=a^{-1}\bgamma\cdot \nabla |\nabla p|^2  +  |\nabla p|^2 \bgamma\cdot \nabla  a^{-1} + (m-1)\bgamma\cdot\nabla (a^{-1}pw)\\
&= 2a^{-1}\sum_{i,j=1}^d \bgamma_i p_{x_j}p_{x_i,x_j}  +  |\nabla p|^2 \bgamma\cdot \nabla  a^{-1} + (m-1)\bgamma\cdot\nabla (a^{-1}pw)\\
&\geq -2a^{-1}\sum_{i,j=1}^d \left(\frac{|p_{x_i,x_j}|^2}{8} + 2\bgamma_i^2 p_{x_j}^2\right) +  |\nabla p|^2 \bgamma\cdot \nabla  a^{-1} + (m-1)\bgamma\cdot\nabla (a^{-1}pw) \\
&\geq -\frac{1}{4a}\sum_{i,j=1}^d |p_{x_i,x_j}|^2 - C   |\nabla p|^2 + (m-1)\bgamma\cdot\nabla (a^{-1}pw) ,
\end{align*}
where  the final two inequalities follow from Young's inequality and Assumption \ref{Assump: a,b,Phi}, respectively.

We now put together all of our bounds for the terms on the righthand side of \eqref{eq:wt}: namely, the previous bound, \eqref{eq:bd Del pt}, and the estimate $\bgamma_t\cdot \nabla p\geq -C|\nabla p|$. We obtain,
\begin{align*}
w_t&\geq  \frac{3}{4a}\sum_{i,j=1}^d |p_{x_i,x_j}|^2 +2a^{-1}\nabla p\cdot \nabla w   -C|\nabla p|- C |\nabla p|^2  +(\tPhi(p))_t+ \\
&\quad \quad \quad  +(m-1)\left( \bgamma\cdot\nabla (a^{-1}pw) +\Delta \left(a^{-1}pw\right)\right).
\end{align*}
Next we use the fact that $\sum_{i,j=1}^d |p_{x_i,x_j}|^2\geq  \frac{1}{d} |\Delta p|^2$, along with  \eqref{eq:w Del p}, to find,
\begin{align*}
 \frac{3}{4a}\sum_{i,j=1}^d |p_{x_i,x_j}| ^2&\geq  \frac{3}{4d\sup a}(w-\bgamma\cdot \nabla p - \tPhi(p))^2\\
 &= \frac{3}{4da}\left(w^2-2w(\bgamma\cdot \nabla p +\tPhi(p))+ (\bgamma\cdot \nabla p +\tPhi(p))^2 \right)\\&
 \geq \frac{3}{4da}\left(w^2 -\frac{1}{2}w^2 - C|\nabla p|^2 -C|w| -C\right)
\\&  = \tilde{C}w^2 - C|w| -C|\nabla p|^2-C,
\end{align*}
where to obtain the second inequality we used Young's inequality, as well as that, according to Assumption \ref{Assump: a,b,Phi} and Lemma \ref{lem:basic}(\ref{item:Linfty and cpt supp}), we have $\|\bgamma\|_{L^\infty} +\|\tPhi(p)\|_{L^\infty}\leq C$. We  denote  $\tilde{C}=\frac{3}{8d\sup a}$. Furthermore, we have 
\[
\tPhi(p)_t = \tPhi_t(p) + \tPhi_p(p)p_t \geq -C -C|p_t|,
\]
where we've again used Assumption \ref{Assump: a,b,Phi} and Lemma \ref{lem:basic}(\ref{item:Linfty and cpt supp}) to obtain that $\Phi_t(p)$ and $\tPhi_p(p)$ are uniformly bounded independently of $m$.

Combining the three previous estimates, and using  $|\nabla p|\leq (|\nabla p|^2+1)/2$, yields 
\begin{align*}
w_t&\geq \tilde{C}w^2 - C|w| +2a^{-1}\nabla p\cdot \nabla w   -C\left(1+  |\nabla p|^2  +|p_t| \right)+(m-1)\left( \bgamma\cdot\nabla (a^{-1}pw) +\Delta \left(a^{-1}pw\right)\right),
\end{align*}
which is exactly 
  \eqref{eq:w final}.
\end{proof}

With this lemma in hand, we proceed with:
\begin{proof}[Proof of Proposition \ref{Prop: AB-est L3}]
We multiply \eqref{eq:w final} by $-|w|_-$, yielding
\begin{align*}
\frac{1}{2}(|w|_-^2)_t &= -w_t|w|_-\leq \\
&-\wm \tilde{C}w^2 +  C|w|\wm -\wm 2a^{-1}\nabla p\cdot \nabla w   +C(1+  |\nabla p|^2  +|p_t| )\wm\\&\quad \quad -(m-1)\left( \bgamma\cdot\nabla (a^{-1}pw)\wm +\Delta \left(a^{-1}pw\right)\wm\right)
\\&= -\tilde{C}\wm^3 +  C\wm^2 +2a^{-1}\wm\nabla p\cdot \nabla \wm  +C\left(1+  |\nabla p|^2  +|p_t| \right)\wm\\&\quad \quad +(m-1)\left( \bgamma\cdot\nabla (a^{-1}p\wm)\wm +\Delta \left(a^{-1}p\wm\right)\wm\right),
\end{align*}
where to obtain the last equality we used, for instance,  $-w|w|_- = \wm^2$ and $ -\wm \nabla \wm= \wm \nabla w$. 
 Now we integrate over  $\Omega_T=B_{CT} \times (0,T)$ (we recall that the support of $p(\cdot, t)$ is contained in $B_{CT}$ for almost every  $t \in [0,T]$).  
 We obtain,
 \begin{align}
 \label{eq:wm3}
 -\int_{B_CT} \frac{|w_0|_-^2}{2}\, dx \leq -\tilde{C}\iint \wm^3 +   C\iint \wm^2  +I_1 +I_2+I_3+I_4, 
 \end{align}
 where 
 \begin{align*}
 &I_1 = \iint 2a^{-1}\wm\nabla p\cdot \nabla \wm, \quad I_2= C\iint \left(1+  |\nabla p|^2  +|p_t| \right)\wm, \\
& I_3=(m-1)\iint \bgamma\cdot\nabla (a^{-1}p\wm)\wm, \text{ and }\quad I_4=(m-1)\iint\Delta \left(a^{-1}p\wm\right)\wm.
 \end{align*}
We begin with $I_4$: we integrate by parts, perform standard manipulations, and finally integrate by parts again in the first term:
\begin{align*}
I_4 &= -(m-1)\iint \nabla  \left(a^{-1}p\wm\right)\cdot\nabla \wm\\
&= -(m-1) \iint \wm \nabla \wm\cdot \nabla  \left(a^{-1}p\right) -(m-1) \iint a^{-1}p\nabla  \wm\cdot \nabla \wm\\
&=   -(m-1) \iint \frac{1}{2}\nabla (\wm^2) \cdot\nabla \left(a^{-1}p\right)  -(m-1)   \iint a^{-1}p|\nabla  \wm|^2\\
&=    \frac{m-1}{2}\iint \wm^2 \Delta \left(a^{-1}p\right) -(m-1)  \iint a^{-1}p|\nabla  \wm|^2 \leq  \frac{m-1}{2}\iint \wm^2 \Delta \left(a^{-1}p\right).
\end{align*}
Using \eqref{eq:w Del p}, we find,
\begin{align*}
\Delta \left(a^{-1}p\right) &= a^{-1}\Delta p + 2\nabla a^{-1}\cdot \nabla p + p\Delta a^{-1}\\&
= a^{-1}(w - \bgamma\cdot \nabla p -\tPhi(p)) + 2\nabla a^{-1}\cdot \nabla p + p\Delta a^{-1}\\
&=a^{-1}w + (2\nabla a^{-1} -a^{-1}\bgamma)\cdot \nabla p +p  \Delta a^{-1} - a^{-1}\tPhi(p).
\end{align*}
 Using this in the bound on $I_4$, and recalling $w\wm^2 = -\wm^3$, we find,
 \begin{align*}
I_4 &\leq -  \frac{m-1}{2}\iint a^{-1}\wm^3 + \frac{m-1}{2}\iint \wm^2(2\nabla a^{-1} -a^{-1}\bgamma)\cdot \nabla p+\\
&\quad  +\frac{m-1}{2}\iint  \wm^2(p  \Delta a^{-1} - a^{-1}\tPhi(p)).
 \end{align*}
 The key point is that the first term is nonpositive, and, moreover, is multiplied by (a factor of) $m-1$. We will ``absorb" ``troublesome" terms into this one. We start with the second term above: Young's inequality yields,
 \[
 \wm^2(2\nabla a^{-1} -a^{-1}\bgamma)\cdot \nabla p\leq \frac{A}{8} \wm^{3} +C |(2\nabla a^{-1} -a^{-1}\bgamma)\cdot \nabla p|^{3}
 \]
 where we denote $A=\inf a^{-1}$, which is positive by Assumption \ref{Assump: a,b,Phi}. Since $|(2\nabla a^{-1} -a^{-1}\bgamma)|$ and $|p  \Delta a^{-1} - a^{-1}\tPhi(p)|$ are uniformly bounded, independently of $m$, the previous line and Proposition \ref{Prop: L3 for Dp} yield,
  \begin{align*}
I_4 &\leq -  (m-1)\frac{7A}{16}\iint \wm^3 +C(m-1) \iint \wm^2+(m-1)C .
\end{align*}

 We continue with $I_3$. We have,
 \begin{align*}
 \bgamma\cdot\nabla (a^{-1}p\wm)\wm &= \bgamma \cdot \nabla (a^{-1} p)\wm^2 + a^{-1}p \bgamma\cdot\nabla \wm \wm\\
 & =   \bgamma \cdot \nabla (a^{-1} p)\wm^2 +\frac{1}{2} a^{-1}p \bgamma\cdot\nabla (\wm^2),
 \end{align*}
 so that, upon integrating by parts, and then using Young's inequality, we obtain,
 \begin{align*}
 I_3& = (m-1)\iint \wm^2 \left(\bgamma \cdot \nabla (a^{-1} p)- \frac{1}{2}\nabla\cdot(a^{-1} p \bgamma)\right)\\
& \leq \frac{A}{16}(m-1)\iint \wm^3 + C(m-1)\iint \left|\bgamma \cdot \nabla (a^{-1} p)- \frac{1}{2}\nabla\cdot(a^{-1} p \bgamma)\right|^{3}\\
& \leq \frac{A}{16}(m-1)\iint \wm^3 + C(m-1),
 \end{align*}
 where the final inequality follows, after a number of standard manipulations, from Proposition \ref{Prop: L3 for Dp} and Assumption \ref{Assump: a,b,Phi}. 

 For $I_2$, we use Young's inequality twice to obtain,
 \begin{align*}
 I_2&= C\iint \left(1+  |\nabla p|^2  +|p_t| \right)\wm\\
 &\leq C\iint \wm +C\iint |\nabla p|^3+\tilde{C}\iint \wm^3 + C\iint |p_t|^2 +C\iint \wm^2\\
 &\leq C\iint \wm +\tilde{C}\iint \wm^3  +C\iint \wm^2 +C,
 \end{align*}
where the final inequality follows from Proposition \ref{Prop: L3 for Dp} and Lemma \ref{lem:basic}(\ref{item:pt}) together with H\"{o}lder's inequality. Here again we have $\tilde{C} = \frac{3}{8d\sup a}$.

 To obtain an estimate on $I_1$, we note 
 \[
 2a^{-1}\wm \nabla p\cdot \nabla \wm = a^{-1}\nabla p\cdot \nabla (\wm^2).
 \]
Using this in the definition of $I_1$ and then integrating by parts  yields,
 \begin{align*}
 I_1&=-\iint \nabla\cdot( a^{-1}\nabla p) \wm^2 = -\iint a^{-1}\wm^2 \Delta p +\nabla (a^{-1})\cdot \nabla p \wm^2.
 \end{align*}
 Using \eqref{eq:w Del p}, followed by Young's inequality, we find,
  \begin{align*}
 I_1&=-\iint a^{-1}\wm^2 (w- \bgamma\cdot \nabla p -\tPhi(p)) +\nabla (a^{-1})\cdot \nabla p \wm^2\\
 & =\iint a^{-1}\wm^3 +a^{-1}\wm^2  \nabla p\cdot(\bgamma -\nabla (a^{-1}))+ \wm^2a^{-1}\tPhi(p)\\
 &\leq \iint a^{-1}\wm^3 +\iint\wm^3 +C\iint |a^{-1}\nabla p\cdot(\bgamma -\nabla (a^{-1}))|^3+ \iint\wm^2a^{-1}\tPhi(p)\\
 &\leq (\sup a^{-1} +1)\iint \wm^3 +C\iint\wm^2 +C,
  \end{align*}
  where the final inequality follows from Proposition \ref{Prop: L3 for Dp}, Lemma \ref{lem:basic}, and Assumption \ref{Assump: a,b,Phi}.

  We now put everything together. By Assumption \ref{Assump: initial data conditions}, the lefthand side of \eqref{eq:wm3} is bounded uniformly in $m$. Thus,  \eqref{eq:wm3} and the estimates on the terms $I_j$, $j=1,...,4$, imply, for $m\geq 2$,
  \begin{align*}
 -C &\leq -\tilde{C}\iint \wm^3 -   C\iint \wm^2  +\left((\sup a^{-1} +1) \iint \wm^3 +C\iint\wm^2 +C\right)\\
 &+\left(C\iint \wm +\tilde{C}\iint \wm^3  +C\iint \wm^2 +C\right)+ \left( \frac{A}{16}(m-1)\iint \wm^3 + C(m-1)\right)\\
 &+\left(-  (m-1)\frac{7A}{16}\iint \wm^3 +C(m-1) \iint \wm^2+(m-1)C\right) \\
  &= (-(m-1)\frac{6A}{16} +\sup a^{-1} +1)\iint \wm^3 + C(m-1)\iint\wm^2 + C\iint\wm + C(m-1).
 \end{align*}
Recall that the integral is taken over $B_{CT}\times (0,T)$. Thus, H\"{o}lder's inequality yields,
\[
((m-1)\frac{6A}{16} -(\sup a^{-1} +1))\iint \wm^3\leq C(m-1)\left(\iint\wm^3\right)^{\frac{2}{3}} + C\left(\iint\wm^3\right)^{\frac13} + C(m-1).
\]
  Therefore, for $m$ as in \eqref{eq:AB m}, 
we obtain 
\[
\iint \wm^3\leq C\left(\iint\wm^3\right)^{\frac{2}{3}} + C\left(\iint\wm^3\right)^{\frac13} + C,
\]
which implies the first estimate in \eqref{eq:prop wm3}, as desired.

To obtain the second  estimate in \eqref{eq:prop wm3}, we use the expression   \eqref{eq:w Del p} for $w$, as well as  Lemma \ref{lem:basic}, and Assumption \ref{Assump: a,b,Phi}, to find,
\[
\iint |\Delta p| \leq \iint |w| +\iint |\bgamma \cdot \nabla p+\tPhi(p)| \leq \iint |w| + C.
\]
Next, we note
\begin{align*}
\iint |w| &= 2\iint \wm +\iint w  \leq C\left(\iint \wm^3\right)^{1/3} +\iint w  \leq C+ \iint w 
\end{align*}
  where we've used H\"{o}lder's inequality, followed by  the first estimate in \eqref{eq:prop wm3}. Now, recalling the expression 
  \eqref{eq:w Del p} for $w$, and then using the fact that $p$ is compactly supported, yields,
  \[
  \iint w = \iint\Delta p + \bgamma \cdot \nabla p+\tPhi(p)= \iint  \bgamma \cdot \nabla p+\tPhi(p) \leq C,
  \]
where the final inequality follows from Lemma \ref{lem:basic} and Assumption \ref{Assump: a,b,Phi}. Putting the previous three estimates therefore yields the second  estimate in \eqref{eq:prop wm3}.

\end{proof}

\section{Proof of the main result}
\label{s:main pf}

We have now gathered all of the ingredients for:

\begin{proof}[Proof of Theorem \ref{Prop: main prop}]

 Lemma \ref{lem:basic} 
yields that $u_m$, $p_m$ are uniformly bounded in $W^{1,1} (Q_T)$. By the Rellich-Kondrachov theorem, there exist $u_\infty,p_\infty\in L^1(Q_T)$ such that, up to a subsequence, $u_m$, $p_m$ converge to $u_\infty$, $p_\infty$ strongly in $L^1  (Q_T)$. In addition, we have that the $u_m$, $p_m$, converge to $u_\infty$, $p_\infty$ almost everywhere on $Q_T$. The estimates of Proposition \ref{Prop: L3 for Dp} and Proposition \ref{Prop: AB-est L3}, and an argument identical to that of \cite[Lemma 3.6]{david2021convective},  imply  
 $\nabla p_m \rightarrow \nabla p_\infty$ strongly in $L^2 (Q_T)$.

Using the definition of $p_m$ and rearranging yields,
\begin{align*}
   \frac{u_m}{b} p_m = \LP \frac{m-1 }{m} \RP ^{1/(m-1)} p_m ^{m / (m-1)}. 
\end{align*}
Taking the limit $m\rightarrow \infty$ yields $\frac{u_\infty}{b} p_\infty = p_\infty $ holds a.e., or equivalently, 
\begin{align*}
    \LP 1 - \frac{u_\infty }{b} \RP p_\infty = 0 \text{ a.e.},
\end{align*}
as desired. In addition, the uniform boundedness and $L^1$ convergence of $u_m$, $p_m$, as well as the previous line, imply, for any bounded function $\phi$,
\begin{align}
\label{eq:HS graph in proof}
\lim_{m\rightarrow \infty} \iint_{Q_T} \left|\left(p_m  \frac{u_m}{b} - p_\infty \right)\phi \right|   = 0.
\end{align}.

Next, we observe that the equality
\[
\frac{u_m}{b}\nabla p_m = (m-1) p_m \nabla \frac{u_m}{b},
\]
together with the uniform boundedness of $u_m$ and integrability of $\nabla p_m$ of Lemma \ref{lem:basic}, implies, 
\begin{equation}
    \label{eq:lim pm nabla um}
\lim_{m\rightarrow\infty }\iint_{Q_T} \left| p_m\nabla \frac{u_m}{b}\right| = 
\lim_{m\rightarrow\infty } \frac{1}{m-1}\iint_{Q_T} \left| \frac{u_m}{b} \nabla p_m \right| =0.
\end{equation}

Let $\zeta\in C^\infty_c(\R^d\times [0,T))$; we shall pass to the limit in the weak formulation 
\begin{equation}
    \label{eq:weak 1}
\iint _{Q_T} u_m \partial _t \zeta - \frac{u_m}{a} \nabla p_m \cdot \nabla \zeta + \frac{u_m}{a} \Phi (x,t,p_m) \zeta \, dx\, dt = - \int _{\R^d} u_m^0 (x) \zeta (x,0)\, dx
\end{equation}
of \eqref{eqn: PME}. Integrating by parts in the second term on the lefthand side yields
\begin{align*}
-\iint_{Q_T}\frac{u_m}{a} \nabla p_m \cdot \nabla \zeta &= \iint_{Q_T} p_m\nabla \cdot \left(\frac{u_m}{b}\frac{b}{a} \nabla \zeta\right) \\
&= \iint_{Q_T} p_m \nabla\frac{u_m}{b} \cdot \left(\frac{b}{a} \nabla \zeta\right) + \iint_{Q_T} p_m \frac{u_m}{b}  \nabla \cdot\left(\frac{b}{a} \nabla \zeta\right).
\end{align*}
Taking the limit and using \eqref{eq:lim pm nabla um}, followed by \eqref{eq:HS graph in proof}, gives,
\[
\lim_{m \rightarrow\infty} -\iint_{Q_T}\frac{u_m}{a} \nabla p_m \cdot \nabla \zeta = \lim_{m \rightarrow\infty}\iint_{Q_T} p_m \frac{u_m}{b}  \nabla \cdot\left(\frac{b}{a} \nabla \zeta\right)= \iint_{Q_T} p_\infty \nabla \cdot \left(\frac{b}{a} \nabla \zeta\right).
\]
Upon taking $m\rightarrow \infty$ in \eqref{eq:weak 1}, the previous line, 
the convergence properties of $u_m$ and $p_m$, and Assumption \ref{Assump: initial data conditions} yield,
\[
\iint_{Q_T} u_\infty\partial_t \zeta  - \iint_{Q_T} \nabla p_\infty  \cdot \left(\frac{b}{a} \nabla \zeta\right) + \frac{u_\infty}{a} \Phi (x,t,p_\infty) \zeta \, dx\, dt = - \int _{\R^d} u^0 (x) \zeta (x,0)\, dx,
\]
which means $(u_\infty, p_\infty)$ satisfies \eqref{eq:limiting eq} in the weak sense.

Next we establish the complementarity relation  \eqref{eq:complementarity}. To this end, the PDE \eqref{eqn: PME Pressure 2} for the pressure $p_m$ yields, for any 
 $\zeta\in C^\infty_c(Q_T)$,
\begin{align*}
        \frac{1}{m-1} \iint _{Q_T} \LP\partial _t p_m- \frac{\vert \nabla p_m \vert ^2}{a} \RP \zeta 
        &= -\iint _{Q_T} \nabla \left(\zeta\frac{p_m}{b}\right) \cdot
        \frac{b}{a}\nabla p_m + \iint _{Q_T} p_m\frac{\zeta}{a} \LP \Phi (x,t,p_m) -  \partial _t b \RP .
            \end{align*} 
The convergence of $\nabla p_m$ in $L^2$, as well as the other boundedness and convergence estimates previously established, allow us to take $m\rightarrow \infty$ and obtain,
\[
0 = -\iint _{Q_T} \nabla \left(\zeta\frac{p_\infty}{b}\right) \cdot
        \frac{b}{a}\nabla p_\infty + \iint _{Q_T} p_\infty\frac{\zeta}{a} \LP \Phi (x,t,p_\infty) -  \partial _t b \RP ,
\]
as desired. 

\end{proof}

\section{Conflicts of interest statement}
The authors have no conflicts of interest.

\bibliography{reference.bib}
\bibliographystyle{plain}

\end{document}